\documentclass[12pt]{amsart}
\usepackage{amsmath,amsthm,amscd,amsfonts,amssymb}
\usepackage[all]{xy}

\begin{document}
   \baselineskip   .7cm
\newtheorem{defn}{Definition}[section]
\newtheorem{thm}{Theorem}
\newtheorem{propo}[defn]{Proposition}
\newtheorem{cor}[defn]{Corollary}
\newtheorem{lem}[defn]{Lemma}
\theoremstyle{remark}
\newtheorem{rem}{Remark}[section]
 \newtheorem{ex}[defn]{Example}
\renewcommand\o{{{\mathcal O}}}
\newcommand\s{\sigma}
\newcommand\w{\widehat}
\newcommand\cal{\mathcal}

\title [Annihilators for cusp forms of weight $2$ and level $4p^m$]
{Annihilators for cusp forms of weight $2$ and level $4p^m$}

\author[A. \'Alvarez]{A. \'Alvarez${}^*$     }
\address{Departamento de Matem\'aticas  \\
Universidad de Salamanca \\ Plaza de la Merced 1-4. Salamanca
(37008). Spain.}

\thanks{$*$  Departamento de Matem\'aticas.
Universidad de Salamanca. Spain.}

\maketitle

  \begin{abstract} We obtain   operators,  given essentially by formal sums of Hecke operators,  that annihilate
  spaces of cusp forms of weight $2$ for $\Gamma_1( p^m)\cap
  \Gamma(4)$, whose dimensions will be specified.
  Moreover, we obtain      the principal part  ($\mathrm{mod} \, p$), over the
  cusps,
  of certain meromorphic modular functions of level $4p^m$.

  \end{abstract}

\tableofcontents


\section{Introduction}

{\bf{Previous notation:}}
In this this work $p$ is a prime integer $\neq 2,3$. We fix $\zeta_{4p^m}$,   a primitive $ {4p^m}$-root of unity.
 Let $\mathfrak p$ be a place   over $p$  in ${\Bbb Z} [\zeta_{4p^m}]$ and $R_{\mathfrak p}$  is the local ring ${\Bbb Z} [\zeta_{4p^m}]_{(\mathfrak p)}$; $k$ and $K$ are its residual field and fraction field, respectively.
The cardinality of $k$ is $p^\delta$, where $\delta=1$ when $4\vert
(p-1) $, and $\delta=2$ in  other cases.
If $X$ is a curve over $\mathrm{Spec}(R_{\mathfrak p})$, then
$ X_k$ denotes $X\otimes_{R_{\mathfrak p}}k$. Let us denote by
$\mathrm{Fr}$   the $p^\delta$-Frobenius morphism. We denote by $A^0_m$   the $p^m$-cyclic subgroup of $({\Bbb Z}/p^m)^{\oplus 2}$ generated by $(1,0)$ and  ${A^0_i}\subset A_m^0 $ its $p^i$-cyclic subgroup. Let $H/L$ be a Galois extension; the group $Aut_L(H)$ is called the Galois group of $H/L$.

\bigskip

Part of this  section is  extracted from \cite{DR}, \cite{L} and \cite{M}.
Let us denote by ${\mathfrak H}$  the upper plane of complex
numbers  with $\mathrm{Im}\, z>0$. Let us denote by ${\mathfrak H}^*$   the union
of the  upper plane with $\infty$ and ${\Bbb Q}$. The modular
congruence  subgroups $\Gamma(n)$  and $\Gamma_1(n)$ consist  of the elements  $\gamma \in Sl_2({\Bbb Z})$ such that
$$\gamma\equiv   \left(\begin{matrix} 1   & 0
  \\0 & 1
  \end{matrix}\right)\,\mathrm{mod} \, n \text{,
   }\gamma\equiv   \left(\begin{matrix} 1   & b
  \\0 & 1
  \end{matrix}\right)\,\mathrm{mod} \, n,$$
(where $b$ is arbitrary), respectively. The quotients  ${\mathfrak
H}^*/\Gamma(n)$ and ${\mathfrak H}^*/\Gamma_1(n)$ give   smooth,
compact modular curves over ${\Bbb C}$, usually denoted by $X(n)$
and $X_1(n)$, respectively. In the same way,  if one considers
commensurable subgroups  $\Gamma\subset Sl_2({\Bbb
 Z})$,   one obtains  modular curves $X_\Gamma$.

A  modular form of weight $2$  for $\Gamma(n)$,
 (respectively $\Gamma_1(n)$), is a  holomorphic function
 $g(z):{\mathfrak H}^*\to {\Bbb C} $
 such that $g(\frac{az+b}{cz+d})=(cz+d)^{ 2}g(z)$  for each $\gamma \in \Gamma(n)$, (respectively $ \Gamma_1(n)$), with
$$\gamma=   \left(\begin{matrix} a  & b
  \\c & d
  \end{matrix}\right). $$
  Therefore, if
  $g(z)$ is a modular form for $\Gamma(n)$,
 (respectively $\Gamma_1(n)$) then $g(z+n)=g(z)$, (respectively, $g(z+1)=g(z)$).
 If we denote $q:=e^{2\pi iz}$, then $g(z)$ has
  a $q$-expansion: $g(z)=\sum_{i=r}^\infty a_iq^{i/n}$, (respectively, $g(z)=\sum_{i=r}^\infty a_iq^{i }$), for some $r\in {\Bbb Z}$.

  The open subset    $X^0(n):={\mathfrak H}/\Gamma(n)$   of $X(n)$   has the
 following modu\-lar interpretation: its points are identified with isomorphism classes of elliptic
 curves  $E$  over ${\Bbb C} $, together with  an isomorphism of the
   $n$-torsion group of $E$ with $({\Bbb Z}/n)^{\oplus 2}$.  The points of the closed subset  $X(n)\setminus
 X^0(n)$, (respectively $X_1(n)\setminus
 X_1^0(n)$)
   are called  the cusps   of $X(n)$, (respectively $X_1(n)$).

A cusp form of weight $2$ for  $\Gamma_1(n)$
   is a modular form $g(z)$ for  $\Gamma_1(n)$
   of weight
$2$    that has a zero at each cusp point. Via the identification
$g(z)\to g(z)dz$,  the space of cusp forms for $\Gamma_1(n)$
  is isomorphic to
the space of $1$-holomorphic differentials on $X_1(n)$.

The meromorphic functions on $X(n)$ are called  meromorphic
modular functions of level $n$.

Let us set
$ R_i   \in \Gamma(1)/\Gamma_1(  p^m)$  as
$$R_i\equiv   \left(\begin{matrix} i   & 0
  \\0 & i^{-1}
  \end{matrix}\right) \mathrm{mod} \,\Gamma_1(  p^m) \text{ and } R_i\equiv   \mathrm{Id} \,\mathrm{mod}\,\Gamma (  4)$$
  with $1\leq i < p^m$ and $(i,p)=1$.
  We use the symbol $T_p$ for the  Hecke operator $T_p(\sum_{i=r}^\infty a_iq^{i } )=\sum_{i=r}^\infty a_{ip}q^{i } $ and $T_{p^h}=T_p^h$. Moreover, $R_i$ and $T_p$ operate   over the cusp forms of level $\Gamma_1( p^m)\cap \Gamma (4 )$.

The Deligne-Rapoport model $M(4p^m )$, (c.f. \cite{DR}
 IV, Corollary 3.9.2.) provides us with   proper models of $X(4p^m)$ over ${\Bbb Z}[1/4, \zeta_{p^m}]$.
  There exists an  open  subscheme $M^0(4p^m )\subset M (4p^m )$    with   $$M^0 (4p^m )\otimes_{{\Bbb Z}[1/4, \zeta_{p^m}]} {\Bbb C}= X^0(4p^m ),$$

   Moreover, there exits a moduli scheme    $M_1(4p^n )$   which is proper and flat over ${\Bbb Z}[1/4 ]$ and   is a model for $X(4p^m)/\Gamma_1(p^m)$, (c.f. \cite{DR}
 IV, Proposition 3.10).

 In this work,  we consider  the moduli  $M (4p^n )  $ and $M_1 (4p^m ) $ defined over the category of schemes over $R_{\mathfrak p}$.
We consider the open subscheme $M^0_1(4p^n )\subset  M_1(4p^n )$ such that $M^0_1(4p^n )\otimes_{R_{\mathfrak p}}{\Bbb C}=X_1^0(4p^m )$.

In this article for each $0\leq i \leq d-1$ we shall define a set $H_i$ and an operator $R_h$ for each $h\in H_i$ such that:

 The operator  $\sum _{i=0}^{d-1}
 (\sum_{h \in H_i}
R_{h }   \cdot T_{p^{\delta i}})$   annihilates  a space of cusp forms, of weight $2$ for the modular group  $\Gamma_1( p^m)\cap \Gamma ( 4)$, of dimension $N$, where $N=\mathrm{dim} (\mathrm{Pic}^0_{M_1(4p^m)_k/k})^{ab}$, (c.f. Theorem \ref{oper}). Let $G$ be an algebraic group, smooth and connected over $k$. There exists a  smallest linear algebraic subgroup $L\subset G$ with $G/L$ an abelian variety. $G^{ab}=:G/L$ is called the abelian part of $G$, c.f. \cite{Ro}.

In the second result of this article, we obtain the principal parts
of certain meromorphic modular functions of level $4p^m$  ($ \mathrm{mod}\,  p$) over  $M(4p^m)$. Let    $\mathrm{Spec}(B)=M^0(4p^m)$  and  let $\overline q^{1/4p^m}$  be the reduction  $ \mathrm{mod}\, \mathfrak p$ of the local parameter $q$ for the $\infty$.  We prove in  Theorem \ref{prin} that:

 For each $m\in {\Bbb N}$, there exists an element of $B_k[p(\sigma)^{-1}]$
  such that its principal is
 $\sum _{i=0}^{d-1}
 (\sum_{h \in H_i}
  \frac{1}{ ({\overline
q^{mp^{\delta i}} )^{R^{-1}_{h }}} }).$

 The function field
 of
${ M}_4[1/4]\otimes_{{\Bbb Z}[i]}{k}$ is $k(\sigma)$, (c.f \cite{H} 4.3), and $p(\sigma)\in k[\sigma]$ is   defined in section 4.1.

To prove the above two results, we use    the correspondences  found by
Anderson and Coleman, (c.f.\cite{An},\cite{C}),  which are  trivial
  (up to vertical and horizontal ones) and that are defined over
abelian extensions of   rational function fields.  These
correspondences are given by   Stickelberger's elements for
function fields and they provide proof  of the Brumer-Stark
theorem in the function field case.
To prove  the first result  of this article, we need    the Eichler-Shimura Theorem for $T_p$ over $M_1(4p^m)$.

The results of the first part  of this work is useful for obtaining,  in an algorithmic   way, the $p$-term in the Euler product of the Hasse-Weil $L$-function of $M_1(4p^m)$ (section 4.3), while the second result could be considered as an additive version of the arithmetic problem of explicitly obtaining  generators for principal ideals of global fields.

\section{Preliminaries}

\subsection{Models  for modular curves}\label{mo}

   We consider $M(4p^m )$  and $M_1(4p^m )$ defined over $R_{\mathfrak p}$. Note that $M(n)$ for  small   $n $ is not a scheme.
For the modular interpretation  of these models we use the Drinfeld level structures for elliptic curves, c.f. \cite{KM}.
\begin{defn}(c.f. \cite{KM}, chapter 3) Let $E$ be an elliptic curve over a scheme $S$. A $\Gamma(p^m)$ (respectively $\Gamma_1(p^m)$)-Drinfeld level structure  is an homomorphism of groups $\phi_{p^m}:({\Bbb Z}/p^m)^{\oplus 2}\to E_{p^m}(S) $, (respectively $\phi^1_{p^m}:{\Bbb Z}/p^m\to E_{p^m}(S) $), such that the divisor $\sum_{a\in ({\Bbb Z}/p^m)^{\oplus 2} } \phi_{p^m}(a)$, (respectively $\sum_{a\in {\Bbb Z}/p^m } \phi^1_{p^m}(a)$) gives a direct sum of two cyclic subgroup subschemes of $E$ of order $p^m$ (respectively a cyclic subgroup subscheme  of $E$ of order $p^m$). Note that in the case of $S=\mathrm{Spec}(k)$ with  $char(k)=p$, a  $p^m$-cyclic subgroup subscheme of $E_{p^m}$ could be
  $  \mathrm{Spec}(k[t]/t^{p^m})$  and hence the divisor given by $\phi^1_{p^m} $ could be  $p^m\cdot 0$. When $E_{n}$ is etale over $S$ and $\phi_{n}$ is a  usual $n$-level structure then  $\phi_{n}$   will be denoted by $\iota_{n}$.
\end{defn}

\begin{defn}(c.f. \cite{KM}, chapter 3) A  $\Gamma_1(p^m)$-level structure over an elliptic curve $E$ over $S$ is an object  $(E, C_{p^m}, \phi^1_{p^m} )$, with   $\phi^1_{p^m}$ a   $\Gamma_1 (p^m)$-Drinfeld level structure and $C_{p^m}=\phi^1_{p^m}({\Bbb Z}/p^m)$. The objects  $(E,    \phi_{p^m} )$ are said to be $\Gamma  (p^m)$-level structures.
\end{defn}

Bearing in mind that $M(4)$ is a fine moduli over $R_{\mathfrak p}$ for the $4$-level structures  (c.f. \cite{DR} IV, Corollaire 2.9), by \cite{KM} 5.5.1, and \cite{DR} IV, we have that
 $M^0(4p^m)$ (respectively  $M^0_1(4p^m)$) are schemes  and they have the following  modular  interpretation:
Let $S$ be a   scheme over $R_{\mathfrak p}$, $M^0(4p^m)(S) $ (respectively  $M^0_1(4p^m)(S) $)  are the sets    $$\{(E,\iota_4 , \phi_{p^m})\} \text{ (respectively  }\{(E,\iota_4, C_{p^m}, \phi^1_{p^m})\}),$$   where $E$ is an elliptic curve over $S$ and $\iota_4$ is a  usual $4$-level structure.

  Let ${A_m}$ be a cyclic subgroup of $({\Bbb
Z}/p^m)^{\oplus 2} $  of cardinality $p^m$. We define  the functor ${\mathcal C}^0_{A_m}(S)$  as the set of objects $(E,\iota_4,C_{p^m}, \phi_{p^m} )$
  such that   $\phi_{p^m}(A_m)=C_{p^m}$  is an etale $p^m$-cyclic subgroup of $E$ over the points of $S$ where $E$ is not supersingular. By \cite{DR} V, 4.8.
  ${\mathcal
C}^0_{A_m}\subset M^0(4p^m)$ gives a  curve   over $R_{\mathfrak p}$.

Let $\{A_m^i\}_{0\leq i\leq l-1  }$ be  the set of $p^m$-cyclic subgroups of $({\Bbb Z}/p^m)^{\oplus 2}$ and $l$  the cardinality of this set. According to    \cite{DR} V, Theorem 4.12,  the decomposition into the reduced
and irreducible components of the scheme $  M(4p^m )_k  $  takes the form
$\cup^{l-1}_{i=0}   ({\mathcal C}_{A^i_m})_k$.  The curves $({\mathcal C}_{A^i_m})_k$ are  smooth over $k$. We denote $ ({\mathcal C}^0_{A^i_m})_k=({\mathcal C}_{Aî_m})_k\cap M^0(4p^m )_k$.

\begin{rem}\label{psi}
Let $A_m^0   $  and ${A^0_i}\subset A_m^0 $ be as in the Introduction. We have an injective morphism  of functors $\Psi_i:({\mathcal C}^0_{A^0_i} )_k\hookrightarrow  M^0_1(4p^m)$.
Let us consider   an object  $(E,\iota_4 , \phi_{p^i} )\in ({\mathcal C}^0_{A^0_i} )_k$ and  the morphism of elliptic curves
$$h_i:E\to E/\phi_{p^i}(A^0_i)\overset{\mathrm{Fr}^{m-i}} \longrightarrow (\mathrm{Fr}^{m-i} )^*(E/\phi_{p^i}(A^0_i)),$$
 where, $\mathrm{Fr}$ denotes the $p$-Frobenius morphism. Thus, we define
 $$\Psi_i  (E,\iota_4  , \phi_{p^i})=(E, \iota_4, \mathrm{Ker }(h_i) , \phi^1_{p^m} ),$$  where  $\phi^1_{p^i}:{\Bbb Z}/p^m\simeq A_m^0 \to E_{p^m}(S)$
  with $\sum_{a=0}^{p^{m-i}}\phi^1_{p^m}(a)=p^{m-i}\cdot 0$ and  ${\phi^1_{p^m}}_{\vert {A^0_i}}={\phi_{p^i}}_{\vert {A^0_i}}$. We have that $\mathrm{Ker }(h_i)=\phi^1_{p^m}({\Bbb Z}/p^m)$.
\end{rem}

 Let $A^0_m,\cdots, A_m^m$ be   elements  for
 the equivalence classes of the quotient set for the action of $\Gamma_1(p^m)$ on   the set of $p^m$-cyclic subgroups  ${A^j_m}$  of $({\Bbb
Z}/p^m)^{\oplus 2} $, where  $A^0_i=A_m^0\cap A^{m-i}_m$.
The decomposition into reduced and irreducible components of $ {M^0_1(4p^m )}_k$ is $\cup^{m}_{i=0} \Psi_i(({\mathcal
C}^0_{A^0_i})_ k) $. This result is  deduced   from both \cite{DR} V, 4.8, 4.11 and 4.12,  and
  \cite{KM} 13.5.6.


\subsection{N\'eron models}
Let $R$ be a discrete valuation ring, $K$ its fraction field,  and $k$ the residual field, which we assume to be perfect. The N\'eron model  for an abelian variety $J_K$ defined over $K$  is a  group scheme ${\mathcal J}$  smooth over $R$,  such that ${\mathcal J}_K\simeq J_K$  and for each scheme ${\mathcal H}$ over $R$ the natural morphism
   $$(*)Hom_R({\mathcal H}, {\mathcal J})\to Hom_K({\mathcal H}_K, {\mathcal J}_K)$$
   is bijective.

   Let $X$ be a   flat, proper   curve over $R$ with $X_K$ smooth and geometrically irreducible over $K$.       We assume that $X_k$ has an irreducible component of  geometric multiplicity $1$.  In the above conditions, the identity component for
   the  N\'eron model ${\mathcal J}$ for $\mathrm{Pic}^0_{X_K/K}$  is $ \mathrm{Pic}^0_{\overline X/R}$, where $\overline X\to X$ is a desingularization for $X$. (C.f. \cite {BLR}, 9.5 Theorem 4 and 9.7 Theorem 1).

\begin{defn}Let  $G$, $H$ be two abelian group schemes  over $R$. $f:G\to H$ is said to be a quasi-isogeny if there exists a   homomorphism   $g:H\to G$, with $g\cdot f=[l\cdot]$, for some $l\in {\Bbb N}$.\end{defn}

   \begin{lem}Let ${\mathcal J}$, $ {\mathcal L}$ be the N\'eron models for the abelian varieties   $J_K$ and $ L_K$, respectively. Let $f:J_K\to L_K$ be an  isogeny and ${\mathfrak f}: {\mathcal J}\to {\mathcal L}$ the morphism associated with $f$ by the   bijection (*). We have that
   ${\mathfrak f}_k:{\mathcal J}_k\to {\mathcal L}_k$ is a  quasi-isogeny and ${\mathfrak f}^{ab}_k:{\mathcal J}^{ab}_k\to {\mathcal L}^{ab}_k$ is an isogeny. We denote by  ${\mathfrak f}_k^{ab}$   the homomorphism given by ${\mathfrak f}_k$ between the abelian parts of ${\mathcal J}_k$ and ${\mathcal L}_k$.
   \end{lem}
   \begin{proof} Since $f$ is an isogeny  ${\mathfrak f} $ is a quasi-isogeny  and therefore  ${\mathfrak f}_k$ is also a quasi-isogeny. The remaining assertion is proved because over abelian varieties quasi-isogenies are isogenies.
   \end{proof}

\begin{lem}\label{ab}Let $g$ be an endomorphism of an abelian variety $J_K$.   Let $ {\mathfrak g}_k^{ab}: \mathcal J^{ab}_k\to \mathcal J^{ab}_k$ be the endomorphism given by ${\mathfrak g}_k$. If $\mathrm{dim}\,\mathrm{Ker }({\mathfrak g}_k^{ab})\geq r$ then $dim \mathrm{Ker } (   g  )\geq r$. Here, we assume that $char(K)=0$.
\end{lem}
\begin{proof} By \cite{BLR} 7.1, Corollary 6, $\mathrm{Ker } ({ g} )$ and $ \mathrm{Im} (  g   )$ admit N\'eron models ${\mathcal N}$ and ${\mathcal I}$, respectively.
Because ${\mathcal N}$ is flat over $R$ we have
$ \mathrm{dim}\, \mathrm{Ker } (  g) =\mathrm{dim}\,{\mathcal N}_k$. Thus, to prove the Lemma it suffices to prove that $  {\mathcal N}^{ab}_k$ is  isogenic to $\mathrm{Ker }({\mathfrak g}_k^{ab})$.

We have the complex ${\mathcal N}\to {\mathcal J}\overset {\mathfrak g} \to {\mathcal I}.  $
Moreover, $J_K$ is   isogenous to $\mathrm{Ker } (  g  ) \times \mathrm{Im}\, ( g  )$ so by the above Lemma ${\mathcal J}$ and ${\mathcal J}_k$ are quasi-isogenous to
${\mathcal N}\times {\mathcal I} $  and ${\mathcal N}_k\times {\mathcal I}_k $, respectively and
${\mathcal J}^{ab}_k$ is isogenous to ${\mathcal N}^{ab}_k\times {\mathcal I}^{ab}_k $. Thus,  we deduce
an exact sequence,  up to isogenies
$$0\to {\mathcal N}^{ab}_k\to{\mathcal J}^{ab}_k\overset{{\mathfrak g}^{ab}_k}\to {\mathcal I}^{ab}_k\to 0$$
and therefore $\mathrm{Ker } ({\mathfrak g}_k^{ab})$ is isogenic to ${\mathcal N}^{ab}_k $.
\end{proof}

\section{The Eichler-Shimura Theorem}

\subsection{Convention for correspondences}\label{conv}
  Let us
consider a proper smooth and geometrically irreducible  curve   $Y$ over a field $k$.  A
correspondence  $C\subset
 Y\times Y$,    given by a Weil divisor,   defines an endomorphism, denoted by $\tilde C$, over the Jacobian  $J_Y$
 in the following way and convention: Let $R$ be an $k$-algebra  and $L(D)$
a line bundle associated  with a  locally principal
 ideal    $I_D\subseteq \o_{Y_R }$  (i.e: $I_D$ is given by an effective Cartier divisor
 $D$  on $Y_R$),   then, $\tilde C(L(D))$ is the line bundle on $Y_{R }$ with  the Cartier divisor
 given
 (locally over $Y_ R$)  by
 the line bundle obtained from the kernel ideal of the ring morphism:
 $$\phi: \o_{Y_R }\to \frac{  \o_{Y_R }  \otimes_R\o_{Y_R }}{I_D\otimes_R \o_{Y_R }
  +I_C},$$
 with  $\phi(a ):= 1\otimes a $, and $I_C$  is the ideal  associated with  $C$.
 We denote by $C(D)$ the Cartier divisor given by $ \mathrm{Ker} (\phi)$. By linear
 extension,
 one can obtain  $\tilde C(L(D))$ and  $C(D)$, with $D$  not
 effective.

Let $\tau:Y\to Y$ be  a scheme morphism. We denote  $\Gamma(\tau)=\{(x,\tau(x)):x\in Y\}$ and its transpose $^t\Gamma(\tau)=\{( \tau(x),x):x\in Y\}$.

\begin{lem}\label{Fr} Let   $k$ be a finite field of $q$ elements. If we denote by
 $D$ a  Cartier divisor  over $ Y_R  $, then $
 ^t{\Gamma(Fr)}(D) =(\mathrm{Fr}\times \mathrm{Id})^*D$ and $ {\Gamma(\mathrm{Fr})}(D) =(\mathrm{Id}\times \mathrm{Fr})^*D$. Here, $  \mathrm{Fr}$ denotes the $q$-Frobenius morphism and $\gamma^*D$ denotes
 the pullback of $D$ by  a morphism $\gamma$.
 \end{lem}
 \begin{proof} Let $I\subset  \o_{Y_R }$  be an ideal   and $I_{\Delta}\subset  \o_{Y_R }\otimes_R
  \o_{Y_R }$ the diagonal ideal. The ideal $I_\Gamma$  given by  $
 ^t{\Gamma(\mathrm{Fr})}$ is $(\mathrm{Id}_Y\otimes \mathrm{Id}_R)\otimes (\mathrm{Fr}\otimes \mathrm{Id}_R))( I_{\Delta})$. Thus,
 $  I\otimes_R  \o_{Y_R } +I_\Gamma = (\mathrm{Fr}\otimes \mathrm{Id}_R)I  \otimes_R \o_{Y_R }+I_\Gamma.$
 Thus,  $^t{\Gamma(\mathrm{Fr})}(D) = (\mathrm{Fr}\times \mathrm{Id}_R)^*D$ for any Cartier divisor $D$.

To prove the remaining equality we consider
 the composition of the
correspondences $
 ^t\Gamma(\mathrm{Fr}) * {\Gamma(\mathrm{Fr})}=q\Delta$, where $\Delta\subset Y\times Y$  is
 the diagonal subscheme. Moreover,   $(\mathrm{Fr}\times \mathrm{Fr})^*D=D^q$  for each  Cartier divisor  $D$  on $Y_R $ and
  $  ^t{\Gamma(\mathrm{Fr})}[  {\Gamma(Fr)} (D)- (\mathrm{Id}\times \mathrm{Fr})^*D]=0$  and
we conclude, since    $  {\Gamma(\mathrm{Fr})}$ is injective acting on the
Cartier group of divisors of  $ {Y_R }$.
 \end{proof}

\subsection{The Eichler-Shimura Theorem for $T_p$}\label{sh}
We fix an isomorphism of groups  $p/p^{m+1}\simeq {\Bbb Z}/p^m$ and we follow the notation of section \ref{mo}.

  First, we shall define the Hecke correspondence  $T_p$ on  $M^0_1( 4p^n)$.

  Following \cite{R}, we consider the morphisms,
  $$\beta,   \alpha:M^0_1(4p^{m+1}) \to M^0_1(4 p^{m }) $$
   $$\beta(E,\iota_4,C_{p^{m+1}}, \phi^1_{p^{m+1}})=(E/C_p,\iota'_4, C_{p^{m+1}}/C_p,  \overline \phi^1_{p^{m+1}}  ),$$
   $$\alpha(E,\iota_4 ,C_{p^{m+1}}, \phi^1_{p^{m+1}} )=(E ,\iota_4, C_{p^{m } } ,  {\phi^1_{p^{m+1}}}_{\vert p/p^{m+1}}),   $$

   $\iota'_4$ is deduced from the isomorphism  $ E_4\simeq E_4/C_p$,
   $\overline \phi^1_{p^{m+1}}$ is $\phi^1_{p^{m+1 }}$ $ \mathrm{mod} \, p^m$ and $C_{p^{i}}=\phi^1_{p^{m+1}}(p^{m+1-i}/p^{m+1 })$.

  Because $M^0_1(4p^m )$ is an affine scheme it is a separated scheme over $R_{\mathfrak p}$. Moreover,
  $\alpha$ is a finite morphism, and therefore proper, because the finite morphism of forgetting the $p^m$-level structures,  $M ^0(4p^m )\to M ^0(4 )$, factorizes through $\alpha$. Thus, we have the closed subscheme   $$\Gamma^0_p=\{ (\beta(x),\alpha(x))\text{ with } x\in M ^0_1(4p^{m+1} )\}\subset M^0_1(4p^m )\times_{R_{\mathfrak p}} M^0_1(4p^m ).$$
Let $S$ be an $  R_{\mathfrak p} $-scheme. Then, $\Gamma^0_p$ has the following modular interpretation
$$ \Gamma^0_p(S)=\{ (E,\iota_4 ,C_{p^{m }}, \phi^1_{p^{m }}  ),  (\overline E,\overline \iota_4 ,\overline C_{p^{m }}, \overline \phi^1_{p^{m }} )  \}$$
such that there exists an isogeny $\psi:E \to \overline E$ of degree $p$ with $ \overline \iota_4=\psi\cdot   \iota_4$,
$ \overline \phi^1_{p^{m }}=\psi\cdot   \phi^1_{p^{m }}$ and $\psi( C_{p^{m }})=\overline C_{p^{m }}$.

According to \cite{DR}
 IV, Proposition 3.10,  we consider $M_1(4p^m):=(M(4p^m)/G)_{ R_{\mathfrak p}}  $, which is a  proper and flat  scheme   over $ R_{\mathfrak p}$,   where  $$G:=\{\gamma \in Gl_2({\Bbb Z})\text{ such that }
 \gamma\equiv   \left(\begin{matrix} 1   & b
  \\0 & d
  \end{matrix}\right)\,\mathrm{mod} \, p^m\}.$$
 This scheme  is a compactification   of $M^0_1(4p^m)$.

  Now, let $f:\overline{M}\to M_1(4p^m)$ be a desingularization of  $M_1(4p^m)$.
  By \cite{KM}, Theorem 10.12.2,     $\overline{M}$  can be obtained from the  $\Gamma_1(p^m)$-balanced moduli problem over $M(4)$. Moreover, for each automorphism $R_j$, there exists an automorphism  $R_j:\overline{M}\to \overline{M}$, (which we also denote  by $R_j$), such that $R_j\cdot f=f\cdot R_j$.

  The   schematic closure of $\Gamma^0 _p\subset  M_1(4p^{m })\times_{R_{\mathfrak p}} M_1(4p^{m })$ will be denoted by $\Gamma_p$ and  $\overline \Gamma_p:=(f\times f)^{-1}( \Gamma_p)$.

\begin{lem}   $\overline \Gamma_p$  provides
an   endomorphism, which we denote  by $T_p$, of the scheme $\mathrm{Pic}^0_{\overline M/R_{\mathfrak p}}$.\end{lem}
\begin{proof}According to \cite{KM}, 13.5.6,  we are in the conditions of \cite{BLR}, Theorem 1, 9.7 , and hence
$\mathrm{Pic}^0_{\overline M/R_{\mathfrak p}}$ is a smooth scheme over $ \mathrm{Spec}(R_{\mathfrak p})$. It is the identity component of the  N\'eron model for $\mathrm{Pic}^0_{M_1(4p^m)_K/K}=\mathrm{Pic}^0_{\overline M_K/K}$.

To find the endomorphism $T_p$  it suffices to give for each smooth morphism
     $W\to \mathrm{Spec}(R_{\mathfrak p})$  an   endomorphism
   $$(T_p)_W: \mathrm{Pic}^0 _{\overline M/R_{\mathfrak p}}(W)\to \mathrm{Pic}^0_{\overline M/R_{\mathfrak p}}(W).$$

   Let $\pi_i$, $i=1,2$, be the natural projections $\pi_i:\overline M_W\times_W \overline M_W\to \overline M_W$, and $ \overline \pi_i   $ its restrictions to $(\overline \Gamma_p)_W$. Following the convention fixed in the previous section, let $L$ be a line bundle over $\overline M_W $,    $\overline \pi_1^*(L)$ is a line bundle on  $(\overline\Gamma_p)_W$ with Weil divisor $Z$. Therefore,  the pushforward $(\overline \pi_2)_*(Z)$ is a Weil divisor on the regular scheme $\overline M_W $. We define  $(T_p)_W(L)$ as the line bundle on $\overline M_W$ associated with this Weil divisor. Note that $\overline M_W$ is a regular scheme because $\overline M_W\to \overline M$ is smooth and $\overline M$ is regular.
\end{proof}

 For a  morphism  $R_{\mathfrak p}\to {\Bbb C}$,  one can write  down  $T_p$ explicitly in terms of the modular parameter $\tau$.  Let $E_\tau$ be the elliptic curve given by the torus obtained from the lattice $1\cdot {\Bbb Z}+\tau \cdot{\Bbb Z}$, $\tau \in {\Bbb C}$. Let $(E,C_{p^m},\phi_{p^m} )$  be given by the pair $(E_\tau, \frac{   p^{-m}\Bbb Z } {\Bbb Z}, \phi^1_{p^m})$. Here,  $\phi^1_{p^m}$ is given by the natural isomorphism
   ${\Bbb Z}/p^m{\Bbb Z}\simeq p^{-m}{\Bbb Z}/{\Bbb Z} $.  We have:
   $$T_p (E,p^{-m}{\Bbb Z}/{\Bbb Z}, \phi^1_{p^m})=\sum_{i=0}^{p-1}(E_{\frac{i+\tau }{p}}, \frac{   p^{-m}\Bbb Z } {\Bbb Z} , \phi^1_{p^m}).$$

If $f(\tau)$ is a  cusp form of weight $2$  and level $p^m$, with $q$-expansion $\sum_{i=r}^\infty a_iq^i$, then $T_p(f(z))= \sum_{i=r}^\infty a_{ip}q^i$.

{\bf{Notation}} We denote $M:= M_1(4p^{m })$.
We consider    the desingularization   morphisms   $h:(\widetilde  {M_k})_{\mathrm{red}}  \to M_k$ and $\overline h: (\widetilde  {\overline M_k})_{\mathrm{red}}  \to \overline M_k$.

Let us consider $\widetilde f_k:\widetilde {(\overline M_k)}_{\mathrm{red}}\to (\widetilde M_k)_{\mathrm{red}}$, the mosphism induced by $f_k:\overline M_k\to  {  M_k}$. The morphism $\widetilde f_k$ is surjective because $f$ is surjective. Moreover,
we have  $h\cdot \widetilde f_k=f_k\cdot \overline h$.

We denote $(\widetilde {\Gamma}_p)_k :=
 (h\times h)^{-1}({\Gamma_p})_k  $  and $(\widetilde  {\overline \Gamma _p})_k:=(\overline h\times \overline h)^{-1}(   {  \overline \Gamma _p})_k$.

  Note that $ (\Gamma_p^0)_k=(\widetilde \Gamma_p)_k\cap (h\times h)^{-1}( M^0_k\times M^0_k)$.

We have  that $(\widetilde {M}_k)_{\mathrm{red}}=\sqcup_{i=0}^m \Psi_i(({\mathcal C}_{A^0_i})_k)$ and therefore
  $\mathrm{Pic}^0_{(\widetilde {M}_k)_{\mathrm{red}}/k}=\prod_{i=0}^m\mathrm{Pic}^0_{({\mathcal C}_{A^0_i})_k/k}.$ Moreover, $\mathrm{Pic}^0_{({\mathcal C}_{A^0_0})_k/k}=\{0\}$, c.f \cite{KM} 13.5.6.
Now, we prove      the Eichler-Shimura Theorem for $T_p$ and $M_1(4p^m)$.
\begin{lem}\label{es} We have that $ (\widetilde {\Gamma ^0_p})_k(x)= {\Gamma}(\mathrm{Fr})(x)$, where $x$ is a geometric point of $  \Psi_i(({\mathcal C}_{A^0_i})_k)$, ($i>0$),  given by $\Psi_i(  E,  \iota_4 ,     \phi_{p^{i }} )$ with $E$  an ordinary elliptic curve. Moreover, $(\widetilde  {\overline \Gamma _p})_k( (\widetilde{f_k})^{-1} (x))= (\widetilde{f_k})^{-1}({\Gamma}(\mathrm{Fr})(  (x)))$. \end{lem}
  \begin{proof} Let us consider $\Psi_i(  E,  \iota_4 ,     \phi_{p^{i }} )=(  E,  \iota_4 ,  C_{p^{m }},   \phi^1_{p^{m }} )=x$. Since $i>0$ the subgroup $C_p\subset C_{p^{m }}$ is etale because $ \phi_{p^i}(A^0_i)$ is an etale subgroup of $C_{p^{m }}$ of cardinality $p^i$, (see Remark \ref{psi}).
Let us denote  $ (\widetilde {\Gamma ^0_p})_k(x)=\overline x$ with
 $\overline x=( \overline E,  \overline\iota_4 ,  \overline C_{p^{m }},   \overline\phi^1_{p^{m }} )$ and we have an isogeny of degree $p$, $\psi:E\to \overline E$, such that   $ \overline \iota_4=\psi\cdot   \iota_4$,
$ \overline \phi^1_{p^{m }}=\psi\cdot   \phi^1_{p^{m }}$ and $\psi( C_{p^{m }})=\overline C_{p^{m }}$. Since $\psi( C_{p^{m }})=\overline C_{p^{m }}$   we have $\mathrm{Ker}(\psi)\cap C_p=\{0\}$  and thus we have that $\psi$ is the $p$-Frobenius morphism $E\to \mathrm{Fr}^* E=\overline E$, (recall that $C_p$ is etale),    and therefore $\overline x=\mathrm{Fr}(x)$.
The remaining assertion is deduced bearing in mind the above result and
$(\widetilde  {\overline \Gamma _p})_k:=(  h\cdot \widetilde f_k\times h\cdot \widetilde f_k)^{-1}(   {   \Gamma _p})_k$. Recall that, $h\cdot \widetilde f_k=f_k\cdot \overline h$.

 \end{proof}

\begin{lem}\label{T} Let $(T_p)_K:\mathrm{Pic}^0_{M_K/K}\to \mathrm{Pic}^0_{M_K/K}$ be the endomorphism given by $T_p$ and ${\mathcal T}_p:{\mathcal J}\to {\mathcal J}$  the endomorphism given by  $(T_p)_K$ for the N\'eron model ${\mathcal J}$ of $\mathrm{Pic}^0_{M_K/K}$.  We have that  the endomorphism  ${\mathcal T}^0_p:{\mathcal J}^0 \to {\mathcal J}^0$ between the identity component of ${\mathcal J}$ is $T_p$. Recall that $M:=M_1(4p^m)$ and that $k$ is the function field of $R_{\mathfrak p}$.
\end{lem}
\begin{proof} This is deduced from the equalities
$$Hom_R(\mathrm{Pic}^0_{\overline M/R}, {\mathcal J})= Hom_K(\mathrm{Pic}^0_{M_K/K}, \mathrm{Pic}^0_{M_K/K})=Hom_R({\mathcal J}, {\mathcal J}). $$
 These equalities are obtained  because ${\mathcal J}$ is the N\'eron model for $ \mathrm{Pic}^0_{M_K/K}$.
\end{proof}



\section{Annihilators for cusp forms and a  calculation }

\subsection{Annihilators for cusp forms }\label{an}

An elliptic curve  over ${\Bbb F}_p$  in the Legendre form $y^2=x(x-1)(x-\lambda)$ is supersingular if and only if $H(\lambda)=0$, with $$H(\lambda)=(-1)^m\sum^m_{i=0}\left(\begin{matrix}    m
 \\ i  \end{matrix}\right)^2\lambda^i$$  the Deuring polynomial ($m=(p-1)/2$). For more details    see
\cite{H} 13.3.  The function field
 of
${ M}_4[1/4]\otimes_{{\Bbb Z}[i]}{k}$ is $k(\sigma)$, with
 $\lambda=(\sigma+1/2\sigma)^2  $. C.f \cite{H} 4.3.
We denote $p(\sigma):=\sigma^{p-1}
H((\sigma+1/2\sigma)^2) $. Here, $k$ denotes the residual field of $R_{\mathfrak p}$.

Let us consider  $q(\sigma)\in k[\sigma]$. We denote by $K_{q(\sigma)}$  the $q(\sigma){\mathfrak m}_\infty$-ray class field for $k(\sigma)$ that  is the maximal abelian extension, totally ramified, of $k(\sigma)$ and with ideal of ramification given by $q(\sigma){\mathfrak m}_\infty$. Here,  ${\mathfrak m}_\infty$ is the maximal ideal associated with   $\infty, ( 1/\sigma =0)$.

\begin{lem}
We have that the Galois group
of $K_{q(\sigma)}/k(\sigma)$ is isomorphic to  $  (k[\sigma]/q(\sigma))^\times $.
\end{lem}\begin{proof}
From class field theory,  the
abelian extension   $K_{q(\sigma)}/k(\sigma)$ has as Galois group
  $ I  /t_\infty^{\Bbb Z} \cdot {k(\sigma)}^\times \cdot U ({q(\sigma)}),$
  where $I  $ is the idele group of   $k(\sigma)  $, $U ({q(\sigma)} )$ denotes
  the idele subgroup   $\{\mu;  \mu\equiv 1\, mod\, q(\sigma){\mathfrak m}_\infty \}$ and $t_\infty$ is the idele whose entries are $\sigma^{-1}$ at $\infty$ and $1$ elsewhere.

Let $T$ be the zero locus of $q(\sigma){\mathfrak m}_\infty$. Let $I^T $ and  ${U}^T $  the ideles   of $I $
 and ${U}$  outside $T$,  respectively, and
$K^\times(q(\sigma)):={k(\sigma)}^\times \cap U ({q(\sigma)}). $
    We have an isomorphism of groups
$$\alpha: {I^T}/{   K^\times(q(\sigma))\cdot U^T}\to
 {I}/{t_\infty^{\Bbb Z}\cdot k(\sigma)^\times \cdot U(q(\sigma))},$$
where $\alpha(\mu_T)$ is the class of the idele in $ I$, whose
entries are given by $\mu_T$ outside $T  $ and $1$ over the places
of $T$. Moreover, we have an isomorphism
$$\delta: {I^T}/{  K^\times(q(\sigma))\cdot
U^T}\to  (k [\sigma]/p(\sigma))^\times $$
 defined as follows. Let  $\mathfrak h$ be a place on $k(\sigma)$ with $\mathfrak
 h\notin
T$ and    $h(\sigma) $   the irreducible polynomial associated
with $\mathfrak h$. Let $t_\mathfrak h\in I^T$ be   the idele whose entries are $h(\sigma)$ at $\mathfrak h$ and $1$ elsewhere. We define $\delta(t_\mathfrak h)$ as the class of
$h(\sigma)^{-1}$ within $(k
[\sigma]/p(\sigma))^\times $.
\end{proof}

We have that
$\Sigma$ is a totally ramified abelian extension of  $k(\sigma)$ whose ideal of ramification  divides $q(\sigma){\mathfrak m}_\infty$ if and only if $\Sigma$ is a  subextension of $K_{q(\sigma)}$ (see   {\cite{Ha}} Chapter 9).

Let $Y_{q(\sigma)}$ be the smooth, proper and geometrically irreducible
curve over $k$  with function field $K_{q(\sigma)}$. We have the following
Anderson-Coleman result, (c.f. \cite{An}, \cite{C}): there exists
a trivial correspondence, up to vertical and horizontal
correspondences  on $Y_{q(\sigma)}\times Y_{q(\sigma)}$
$$D_k:=\sum _{i=0}^{d-1}
 (  \underset { \mathrm{deg}(r(\sigma))=d-1- i}  {\sum_{r(\sigma)\text{(monic)} } }
\Gamma(\beta_{r(\sigma)}) * \Gamma(\mathrm{Fr}^{ i}))
,$$
with $\mathrm{deg}(q(\sigma))=d$,
  $\beta_{r(\sigma)}:=(\delta^{-1}\cdot \alpha)(r(\sigma))\in Aut_{k(\sigma)}(K_{q(\sigma)})$ and  $r(\sigma)$ are monic polynomials in $\sigma$.
 We   denote  the
composition of correspondences by $*$ and $\mathrm{Fr}$ denotes the $p^\delta$-Frobenius morphism.

\begin{lem} The   extension $\Sigma_{A^0_m}/k(\sigma)$,  given by
 $( {\mathcal C}_{A^0_m})_k\to
M(4 )_k$,   is an abelian extension of group
$({\Bbb Z}/p^m)^\times $. Moreover, $\Sigma_{A^0_m}$ is a
subextension of $K_{q(\sigma)}$, where  $q(\sigma)=p(\sigma)^rk[\sigma]$  for certain
$r\in{\Bbb N}$.
\end{lem}
\begin{proof}
By \cite{DR}  V, Lemma 4.16,   the morphism $  ({\mathcal C}_{A^0_m})_k\to  M( 4)_k$ is     a ramified abelian covering of    group $({\Bbb
Z}/p^m)^{\times}$ and it is
totally ramified  over the pairs $(E, \iota_4) $,  with $E$ a supersingular elliptic curve.
This  means that it  is totally ramified for the
values of $\sigma$ where  $E$ is supersingular. These values are given by the roots of the
Deuring polynomial $\sigma^{p-1}  H(\lambda)$, $\lambda=(\sigma+1/2\sigma)^2$,
(see for example, \cite{H} Chapter 13, {\S}3).
\end{proof}

We have that
$\Gamma(1)$ operates transitively over the irreducible components of
$ M(4p^m)_k=\cup^{l-1}_{i=0}({\mathcal C}^0_{A^i_m})_k $. By fixing $A^0_m$ as in the introduction, there exists $g_i\in \Gamma(1)$, with
$ g_i(({\mathcal C}^0_{A^0_m})_k)=({\mathcal C}^0_{A^i_m})_k  $. The action of the elements $R_l$ on $g_i(({\mathcal C}^0_{A^0_m})_k)$ is given by $g_i\cdot R_l\cdot g_i^{-1}$.

The Galois group of   $\Sigma_{A^0_m}/k(\sigma)$  is identified with
the subgroup of $\Gamma(1)/\Gamma( p^m)$ formed by the classes of
$\{R_i\}_{1\leq i< p^m } \in \Gamma(1)$, where
$$R_i\equiv   \left(\begin{matrix} i   & 0
  \\0 & i^{-1}
  \end{matrix}\right)\mathrm{mod}\,\Gamma(  p^m) \text{, } R_i\equiv  \mathrm{Id}\,\mathrm{mod} \,\Gamma (  4),$$
  with $1\leq i < p^m$ and $(i,p)=1$.

\begin{rem}\label{tri}
Let $\pi: Y_{q(\sigma)}\to ({\mathcal C}_{A^0_m})_k$ be the morphism induced by the inclusion $\Sigma_{A^0_m}\subset K_{q(\sigma)}$. Let   us denote  $\rho:Gal(K_{q(\sigma)}/k(\sigma))\to ({\Bbb Z}/p^{m})^\times,$ the surjective morphism
between the Galois groups of the extensions $K_{q(\sigma)}/k(\sigma) $ and $\Sigma_{A^0_m}/k(\sigma)$.
The pushforward to $({\mathcal C}_{A^0_m})_k\times ({\mathcal C}_{A^0_m})_k$ of the correspondence $D_k$,
   $$(\pi\times \pi)_*(D_k)=\mathrm{deg}(\pi)\cdot\sum_{i=0}^{d-1}
  (\underset { \mathrm{deg}(r(\sigma))=d-1- i}  {\sum_{r(\sigma)\text{(monic)}}}
\Gamma(R_{r(\sigma)}) *{\Gamma}(\mathrm{Fr}^i))$$
   is again a trivial correspondence,    up to vertical and horizontal
correspondences, where  $R_{ {r(\sigma)}}=\rho(\beta_{r(\sigma)})$. Thus, as the ring of   classes of  correspondences is ${\Bbb Z}$-free-torsion, $\mathrm{deg}(\pi)^{-1}\cdot (\pi\times \pi)_*(D_k)$ is also trivial.

Note that  given a proper morphism  $f:Z\to \bar Z$,  between   varieties,  if
$U$ is a $k$-cycle of $Z$, with  $dim \,U=dim\, f(U)$, then the
pushforward  $f_*(U)$  is defined by
$$f_*(U)=[\Sigma_U :\Sigma_{f(U)}]\cdot f(U). $$ Where, $\Sigma_U $ and $\Sigma_{f(U)}$  are the
function fields associated with $U$ and $f(U)$, respectively.

\end{rem}

\begin{rem} Let $Y$ be a   proper curve  over a field $k$  and $ \widetilde {Y}_{\mathrm{red}}$ the normalization of the largest reduced subscheme $Y_{\mathrm{red}}$ of $Y$.  Then, by \cite{BLR}, 9.3 Corollary 11,
   $(\mathrm{Pic}^0_{Y/k})^{ab} \simeq \mathrm{Pic}^0_{\widetilde {Y}_{\mathrm{red}}/k}$. We set $N=\mathrm{dim}  (\mathrm{Pic}^0_{\widetilde {M_1(4p^m)}_k/k})^{ab}$.
\end{rem}

\begin{thm}\label{oper} The operator $$D_K=\sum _{i=0}^{d-1}
 (\underset { \mathrm{deg}(r(\sigma))=d-1- i}  {\sum_{r(\sigma)\text { (monic)}} }
 R_{ r(\sigma) } \cdot  T_{p^{i\delta}})$$ annihilates  a space of dimension $N$ of cusp forms of level $\Gamma(4)\cap \Gamma_1(p^m)$ and weight $2$.
 \end{thm}
 \begin{proof} We have  $(M^0_1(4p^m )_k)_{\mathrm{red}}=\cup^{m}_{i=0} \Psi_i(({\mathcal
C}_{A^0_i})_k)$ (see Remark \ref{psi}) and by the last Remark we have a surjective morphism
$$\mathrm{Pic}^0_{  { M_1(4p^m )_k }  /k }\to \prod^{m}_{i=1}  \mathrm{Pic}^0_{({\mathcal C}_{A^0_i} )_k/k } .$$
Since   $A^0_i\subset A_m^0$, we have  morphisms of curves $\pi_i:({\mathcal C}_{A^0_m})_k\to ({\mathcal C}_{A^0_i})_k$ defined by   $\pi_i(E,\iota_4 , \phi_{p^m} )= (E,\iota_4 , \overline \phi_{p^i} )$, with
$\overline \phi_{p^i}=  {\phi_{p^m}}_{\vert p^{m-i}({\Bbb Z}/p^m)^{\oplus 2}}$. Moreover, because $\pi_i\cdot \mathrm{Fr}=\mathrm{Fr}\cdot \pi_i$  we have that
$(\pi_i
\times \pi_i)_* \Gamma(\mathrm{Fr}) =deg(\pi_i)\cdot  \Gamma(\mathrm{Fr}).$  Here, we have
denoted   the   Frobenius morphism   over ${\mathcal C}_{A^0_m}$ and $ {\mathcal C}_{A^0_i}$ in the
same way. This latter statement also is true for the correspondences associated with the graphics of the  automorphisms $R_j$.

By Remark \ref{tri}, since  $D_k$  annihilates $\mathrm{Pic}^0_{  ({{\mathcal C}_{A^0_m}})_k/k }$, the correspondence $deg(\pi_i)^{-1}\cdot(\pi_i\times \pi_i)_*(D_k)$ also  annihilates $\mathrm{Pic}^0_{  ({{\mathcal C}_{A^0_i}})_k/k }$, with $0\leq i\leq m$.

The Theorem is deduced from the Eichler-Shimura Theorem  (c.f: Lemma \ref{es}),   from Lemma \ref{ab}
 applied to $\mathrm{Pic}^0_{M_1(4p^m)_K/K}$ and to the endomorphism  $D_K$,  knowing that $D_k$  annihilates
 $\prod_{i=1}^m\mathrm{Pic}^0_{  ({{\mathcal C}_{A^0_i}})_k/k }$.

 We bear in mind that the homomorphism $$(\widetilde f_k)^*:(\mathrm{Pic}^0_{\widetilde {M_1(4p^m)}_k/k})^{ab}\to (\mathrm{Pic}^0_{\overline {M  }_k/k})^{ab}$$ has a finite kernel  because
   $\widetilde f_k:    (\widetilde {\overline M_k})_{\mathrm{red}}\to {( \widetilde  {M_1(4p^m)}_k)}_{\mathrm{red}}$ is surjective.
 \end{proof}

\subsection{An explicit calculation}\label{cal}

The morphisms $ ({{\mathcal C}_{A^0_1}})_k\to ({{\mathcal C}_{A^0_0}})_k$ and $ ({{\mathcal C}_{A^0_m}})_k\to ({{\mathcal C}_{A^0_1}})_k$ give  us extensions  $k(\sigma)\subset\Sigma_{A^0_1} $ and
$\Sigma_{A^0_1}\subset\Sigma_{A^0_m} $, respectively, which  are
  abelian extensions of groups $({\Bbb Z}/p )^\times$ and $ {\Bbb Z}/p^{m-1}  $, respectively . One can
obtain the generator of the extension
$k(\sigma)\subset\Sigma_{A^0_1} $,  because it is a totally ramified
extension and  its ramification is given by the polynomial
$p(\sigma)$, (c.f.\cite{I}). Thus, $\Sigma_{A^0_1}=k(\sigma,
p(\sigma)^{1/p-1})$.

According to  \cite{K}, Lemma 2.5   and (5.5),  together with  the following two Lemmas,  allow us to make
explicit calculations for the  Artin-Schreier generators for  the
   extensions
 $\Sigma_{A^0_1} \subset\Sigma_{A^0_m}.$
We denote ${\mathcal P}(a)=a^{p }-a$.
\begin{lem}   The extension $k((t))({\mathcal P}^{-1}( 1/t^n))/k((t))$ is totally ramified over $t=0$.
\end{lem}
\begin{proof} An element $1/x\in {\mathcal P}^{-1}( 1/t)$ satisfies     $x^p+x^{p-1}\cdot t-t=0$. This equation   is a non-singular curve   at the point $(0,0)$.
The Lemma is proved for $n=1$   by considering the support of the differential $k[[t]][x]$-module
$$\Omega_{(k[[t]][x]/x^p+x^{p-1}\cdot t-t)/k[[t]]}.$$
Bearing in mind  the totally ramified extension $k((t^n))\to k((t))$ and $k((t))( {\mathcal P}^{-1}(1/t^n))=k((t^n))( {\mathcal P}^{-1}(1/t^n))\otimes_{k((t^n))}k((t))$, one proves the Lemma  for any $n$.
\end{proof}

\begin{lem} If $k(t, {\mathcal P}^{-1}( r(t)/s(t)))/k(t)$ is an extension such that its ideal of ramification divides the polynomial  $p(t)$, then there exists $l\in {\Bbb N}$ with $r(t)/s(t)=h(t)/p(t)^l+{\mathcal P}(u(t)/v(t))$ and $\mathrm{deg}(h(t))\leq  \mathrm{deg}(p(t)^l)$.
\end{lem}
\begin{proof}  This is deduced from the above Lemma   bearing in mind  the decomposition in simple fractions   of $r(t)/s(t)$ and the fact that ${\mathcal P}$ is additive.
\end{proof}
By the previous Lemma to make explicit calculations   of the Artin-Schreier generators of $\Sigma_{A^0_1}\subset\Sigma_{A^0_m} $ it suffices to calculate $l$ and $h(t)$ for these extensions; $p(t)$ is the polynomial $p(\sigma)$ defined in section \ref{an}.

As example we shall make this calculation precise for $m=2$.  Following the notation of \cite{K} {\S}2 and \cite{Se}  2.2,  the Artin-Shreier generator for $\Sigma_{A_1}\subset
\Sigma_{A_2}$ is the reduction   $\mathrm{mod}\, p$  of
$\psi:=\beta(1)(\frac{a_p-1}{p })=\frac{a_p-1}{p\cdot E_{p-1} }\in   {\Bbb
F}_p((q))$, where $a_p$  is a modular form of weight $p-1$. By
\cite{Se}, 2.2, we have that
$$(\psi)^p-\psi= -\frac{1}{24}\cdot
\beta(1)(\theta^{p-2}(E_{p+1}))=\frac{\theta^{p-2}(E_{p+1})}{E_{p-1}^{3}}\quad (\mathrm{mod}\, p),$$
 with $\theta=q \frac{d}{dq}$ and $\theta^{p-2}(E_{p+1})$    a modular form of weight
 $ 3(p-1)$.

Here, $E_{p+1}$ and $E_{p-1}$ denote  the normalization,  (constant
term = $1$), of the Eisenstein series of weights $p+1$ and $p-1$,
respectively.

The algebra of modular functions is given by the graded ring of
polynomials  ${\Bbb Z}[g_2, g_3]$, where $g_2$ and $g_3$ have grades
 $4$ and $6$, respectively. If
$A(g_2,g_3)(\frac{dx}{2y})^{\otimes p-1}$, ($A(g_2,g_3)\in  {\Bbb
F}_p[g_2, g_3]$), is  the Hasse invariant
 of the universal
elliptic curve in the Weierstrass form  $y^2=4x^3-g_2x-g_3$  then
the class of $E_{p-1}$ in ${\Bbb F}_p[g_2, g_3]$ is $A(g_2,g_3)$.

Since the Deuring polynomial gives the Hasse invariant for the
universal elliptic curve in the Legendre form
$y^2=x(x-1)(x-\lambda)$, we have that
 $\Sigma_{A_2}=\Sigma_{A_1}[\psi]$
with
 $\psi^p-\psi-\frac{h(\sigma)}{p(\sigma)^3}.  $ Here,
$p(\sigma)=\sigma^{p-1}\cdot H(\lambda)$ with $\lambda=(\sigma+1/2\sigma)^2$.

Moreover, because
$\frac{\theta^{p-2}(E_{p+1})}{E_{p-1}^{3}}$ only has poles over
the zeroes of $E_{p-1}(z)dz$ we have that $\mathrm{deg}(h(\sigma))\leq \mathrm{deg}(p(\sigma)^3)(= 3(p-1))$. To calculate  $h(\sigma)$   it suffices to
bear  in mind the expansions    on the parameter  $q^{1/4}$  of
$\sigma$, $\frac{\theta^{p-2}(E_{p+1})}{E_{p-1}^{3}}$  and the
congruence
$$  \frac{\theta^{p-2}(E_{p+1})}{E_{p-1}^{3}} \equiv
 \frac{h(\sigma)}{p(\sigma)^3} \quad (\mathrm{mod}\, \mathfrak p)  .$$

Bearing in mind this latter calculation  and the action of  $(k[\sigma]/p(\sigma)^3
)^\times $ on $\psi$, given by
the class field theory, (c.f. \cite{Ha} Chapter 9),    one can explicitly obtain the
group homomorphism
$$\rho:(k[\sigma]/p(\sigma)^3 )^\times  \to  ({\Bbb
Z}/p^2)^\times.$$

\subsection{A way to calculate the $p$-term   of the Hasse-Weil $L$-function of $M_1(4p^m)$}

To obtain the $p$-term in the Euler product of the Hasse-Weil $L$-function of $M_1(4p^m)$ it suffices to obtain the Euler zeta function of the curve $ ({{\mathcal C}_{A^0_m}})_k$.

 We, now, consider the incomplete $L$-function associated with the Galois extension $K_{q(\sigma)}/k(\sigma)$:
$$\prod_{  x\in
\vert {\Bbb P}^1(k)\vert\setminus T} (1-F_x\cdot t^{deg(x)})^{-1}=\frac{U(t)}{1-p^{\delta}\cdot t},$$
(C.f \cite{An} 1.3), where $U(t)$ is an explicit polynomial in $  {\Bbb Z}[G][t]$; $G= (k
[\sigma]/p(\sigma))^\times $; $F_x$ is the Frobenius element for $x\in
\vert {\Bbb P}^1(k)\vert\setminus T$,   $T$ being the places in $\vert {\Bbb P}^1(k)\vert$ given by the divisors of $q(\sigma){\mathfrak m}_\infty$.

We define $\text{Norm}_{{\Bbb Z}[G]/{\Bbb Z}[H]}( U(t) )$   as the determinant of the  morphism of multiplication by $U(t)$ in the ${\Bbb Z}[t][H] $-module ${\Bbb Z}[t][G] $. $H$ being the kernel of the group morphism
$$\rho:(k[\sigma]/q(\sigma))^\times  \to  ({\Bbb
Z}/p^m)^\times$$
given by the Galois extensions $k(\sigma)\subset \Sigma_{A^0_m}\subset K_{q(\sigma)}$.

Let
$\pi:({{\mathcal C}_{A^0_m}})_k\to {\Bbb P}^1(k)$
be the morphism of smooth curves given by   $k(\sigma)\subset \Sigma_{A^0_m}$.
Bearing in mind that
 $$\text{Norm}_{{\Bbb Z}[G]/{\Bbb Z}[H]} (1-F_x\cdot t^{deg(x)}) =\prod_{ z\in
  \pi^{-1}(x)} (1-F_z\cdot t^{deg(z)})^{-1},$$
we can calculate the  incomplete  $L$-function
$$\prod_{ z\in
\vert ({{\mathcal C}_{A^0_m}})_k\vert\setminus \pi^{-1}(T)} (1-F_z\cdot t^{deg(z)})^{-1}$$
as a faction of polynomials $\frac{V(t)}{W(t)},$
where
$$\frac{V(t)}{W(t)}=\frac{\text{Norm}_{{\Bbb Z}[G]/{\Bbb Z}[H]}( U(t) )}{\text{Norm}_{{\Bbb Z}[G]/{\Bbb Z}[H]}(1-p^{\delta}\cdot t)}.$$

By making $h=1$ for each $h\in H$ in the quotient $\frac{V(t)}{W(t)}$, we obtain  a fraction of polynomials   $\frac{v(t)}{w(t)}$. The Euler zeta function of $ ({{\mathcal C}_{A^0_m}})_k$ is
$$\frac{v(t)}{w(t)}\cdot \prod_{  z\in
 \pi^{-1}(T)} (1-  t^{deg(z)})^{-1} .$$
By taking account the above section we can obtain $H$ in an explicit way  and we can calculate $\frac{v(t)}{w(t)}$.
Moreover, we obtain the finite product $\prod_{  z\in
 \pi^{-1}(T) } (1-  t^{deg(z)})^{-1}  $ by considering   $\pi^{-1}(\infty)$ and the roots of
$q(\sigma)$, which are given by the Deuring polynomial $H(\lambda)$.

\section{The cuspidal principal part ($\mathrm{mod} \,p$) of certain meromorphic modular functions}

\subsection{Line bundles from an adelic point of view}
 Let  $Y$  be a smooth, projective  and geometrically irreducible curve over a field $k$ of $q=p^\delta$ elements; $\Sigma$ denotes its function field. Let $\mathfrak q$ be a place of
$\Sigma$,  $t_\mathfrak q$ a local parameter for $\mathfrak
q$  and $k( \mathfrak q)$ its residual field. For each $k$-algebra $R$, we consider the group
 $(R\otimes k(\mathfrak q))((t_\mathfrak q)) :=(R\otimes k(\mathfrak q))[[t_\mathfrak q]][t_\mathfrak q^{-1}]   $
and the adele group
$${\Bbb A}_\Sigma(R):=\prod'_{\mathfrak q\in
\vert Y\vert} (R\otimes
 k(\mathfrak q))((t_\mathfrak q)),$$
where $\prod'$ refers to  the adeles with a non-trivial principal part
  only on a finite number of places of $\Sigma$. The idele group  $I_\Sigma(R)$  is the group of units of
${\Bbb A}_\Sigma(R)$. Thus,
 $$I_\Sigma(R) =(\prod'_{\mathfrak q\in \vert Y\vert} (R\otimes
 k(\mathfrak q))((t_\mathfrak q)))^\times.$$
Let $U_\Sigma(R)$ be the subgroup
 $ \prod_{\mathfrak q\in \vert Y\vert}(R\otimes k(\mathfrak q))[[t_\mathfrak q]]^\times .$
Let $\Sigma\otimes_k R$ be denoted by $\Sigma_R$. One can
embed  the group  $(\Sigma_R)^\times$  diagonally into $I_\Sigma(R)$.

 The quotient group
$$\frac{I_\Sigma(R)}{\Sigma^\times_R\cdot U_\Sigma(R)} $$
is isomorphic to the class group of line bundles $L$ on $Y\otimes
R$, such that if
$j: \mathrm{Spec}(\Sigma_R) \hookrightarrow Y_R$
is the natural inclusion  then the pullback  $j^*L$  is isomorphic
to the trivial line bundle. See, for example,  \cite{BL}, Section 2, and Lemma 3.4.

If we denote by  $J_Y$  the Jacobian variety of $Y$ and $\mu$ is an idele of degree $1$, we have a
monomorphism of groups
$$\frac{I_\Sigma(R)}{ \mu^{\Bbb Z}\cdot \Sigma^\times_R\cdot U_\Sigma(R)}\hookrightarrow
J_Y(R).$$

This monomorphism is an isomorphism if and only if $R$ is a finite
${k}$-algebra. Recall that $J_Y(R)$  is the group of isomorphism classes of line
bundles over $Y_R$ of degree $0$.

 In the following Lemma  we recall that the tangent
space  of the Jacobian on the zero element  can be identified with
the cohomology group $H^1(Y,\o_Y)$ and by duality with
$H^0(Y,\omega)^\vee$  ($\omega$ is the dualizating sheaf of $Y$):
 \begin{lem}\label{ep} We have an exact sequence of groups
 $$0\to H^1(Y,\o_Y) \to J_Y(k[\epsilon ]) \overset{\epsilon =0}\to J_Y(k
 )\to 0,\quad (\epsilon^2=0).$$
 \end{lem}
\begin{proof} It suffices to take into account that $k[\epsilon]$  is a finite $k$-algebra   and that
$$J_Y(k[\epsilon])\simeq \frac{I_\Sigma(k[\epsilon])}{\mu^{\Bbb Z} \cdot\Sigma^\times_{k[\epsilon]}\cdot U_\Sigma({\Bbb
F}_q[\epsilon])}=\frac{I_\Sigma  }{\mu^{\Bbb Z} \cdot\Sigma^\times \cdot
U_\Sigma}+\epsilon \cdot \frac{\Bbb A}{\Sigma +O_\Sigma},$$
where $O_\Sigma$ denotes  the integer adeles.
\end{proof}

\begin{rem}\label{mu}
Let $D$ be a Cartier divisor given by the class of an
idele
$$  \mu_0+\mu_1\epsilon\in \frac{I_\Sigma(k[\epsilon])}{  U_\Sigma(k[\epsilon])}=\frac{I_\Sigma
}{  U_\Sigma}+\epsilon \cdot \frac{\Bbb A}{ O_\Sigma}.$$
  According to  Lemma \ref{Fr}, $    \Gamma(\mathrm{Fr})(D) $ and $^t\Gamma(\mathrm{Fr})(D) $ are
given by
  $\mu_0+\mu_1\epsilon^q$ and $\mu^q_0+\mu_1^q\epsilon $, respectively. Note
 that
$ \epsilon^{q }=0$.
\end{rem}

\subsection{Correspondences in an additive setting}

We now prove a Lemma that will be used to make explicit calculations in the next section.
 Let $\mathrm{Spec}(B)\subset Y$  be an  open subscheme     and let  $C$ be a correspondence on $Y$ that is
   trivial  over $\mathrm{Spec}(B)\otimes \mathrm{Spec}(
  B)$. We set that  $C$  over $\mathrm{Spec}(B\otimes   B)$ will be  given by
the zero locus of a regular function $H(b,\bar b) \in B\otimes   B$. Let
$\mathfrak q\in \mathrm{Spec}(B) $ be a geometric point
and $t_\mathfrak q$   a local parameter for $\mathfrak q$. For
easy notation, we can assume $\mathfrak q$ to be rational. We
denote by $D^{-r\cdot \mathfrak q}$, $D_{ r\cdot \mathfrak q}$ the
Cartier divisors on $Y\otimes k[\epsilon]$ given by the line bundles associated with the
ideles of $I_\Sigma(k[\epsilon])$, whose entries are  $1-\epsilon
t_\mathfrak q^{-r}$  and $ t_\mathfrak q^{ r}-\epsilon $ at $\mathfrak q$ and $1$ elsewhere,
respectively.

\begin{lem}\label{re} With the   notations and conventions of section \ref{conv}, $C(D^{-r\cdot \mathfrak
q})$  is a Cartier divisor such that over the open subset  $
\mathrm{Spec}(B[\epsilon])\subset Y[\epsilon]$ is given by
 $1+\epsilon \mathrm{Res}_\mathfrak q(t_\mathfrak q^{-r} {\mathrm{d}}\,
\mathrm{log} \,H(t_\mathfrak q,\bar b))\in 1+\epsilon\cdot \Sigma.$

 To calculate this residue,
one considers   $H(b,\bar b)= H(t_{\mathfrak q},\bar b) \in B ((t_\mathfrak q))$ and $ \mathrm{d}
\,
\mathrm{log}\,H(t_\mathfrak q,\bar b)= \frac{\mathrm{d} \,
\mathrm{log} \,H(t_\mathfrak q,\bar b)}{\mathrm{d}\, t_\mathfrak q} \mathrm{d}\, t_\mathfrak q$.
\end{lem}
\begin{proof} We have that  $C(D_{ r\cdot \mathfrak q})$  over
$\mathrm{Spec}(B)$ is given by the kernel ideal of the $B$-algebra
morphism:
$$\phi:B[\epsilon]\to B[\epsilon, t_\mathfrak q]/
 (t_\mathfrak q^{ r}-\epsilon,H(t_\mathfrak q,\bar b) ), \quad \phi(\epsilon)=\epsilon. $$
Since $t_\mathfrak q^{ 2r}=0$ as an element of $$B[\epsilon,
t_\mathfrak q]/
 (t_\mathfrak q^{ r}-\epsilon,H(t_\mathfrak q,\bar b) ),$$ one can assume  that $H(t_\mathfrak q,\bar b)
 =u(t_\mathfrak q)=b_1t_\mathfrak q^{2r-1}+\cdots+b_{2r-1}t_\mathfrak q +b_{2r}$, with
 $b_i\in B$.

   The ideal  $\mathrm{Ker } ( \phi)$ is generated by  $\prod_i
 (\alpha_i^r-\epsilon)$; this product is taken  over the roots of
 $u(t_\mathfrak q)$.

 Because   $C(D^{ -r\cdot \mathfrak q})=C(-r\cdot \mathfrak q)C(D_{ r\cdot \mathfrak
 q})
  $   and because the Cartier divisor  $   C(-r\cdot \mathfrak q) $  is given over $\mathrm{Spec}(B)$ by the element $  \prod_i
  \alpha_i^{-r } $, we have that the Cartier divisor defined over $\mathrm{Spec}(B[\epsilon])$,    $C(D^{ -r\cdot \mathfrak q}) =   C(t_\mathfrak q^{ -r}) C(t_\mathfrak q^{ r}-\epsilon
  ) $  is given by the  element
$   \prod_i
 (1-\epsilon\alpha_i^{-r})=
  1-\epsilon(\sum_i\alpha_i^{-r}) .$

 By using the residues theorem for $1$-differential forms of $\Sigma(t_\mathfrak
 q)$, where the field of constants is $\Sigma$,
 we conclude  because
 $$ \mathrm{Res}_{ \mathfrak q }(t_\mathfrak q^{-r}\mathrm{d} \,
\mathrm{log}\,H(t_\mathfrak q,\bar b))=-\sum_y  \mathrm{Res}_y(t_{\mathfrak q}^{-r}\mathrm{d} \,
\mathrm{log}\,H(t_\mathfrak q,\bar b))=\sum_i
  \alpha_i^{-r}.$$
 The second sum is taken over the maximal ideals  $y$  of $\Sigma[t_\mathfrak q^{-1}]$.
  \end{proof}

\begin{rem}\label{to} It is not hard to prove that if $a\in B((t_\mathfrak q))$ with either $a\in
B$ or $a\in k((t_\mathfrak q))$ then
 $ \mathrm{Res}_{ \mathfrak q }(t_\mathfrak q^{-r}\mathrm{d} \,
\mathrm{log}\,(aH(t_\mathfrak q,\bar b)))= \mathrm{Res}_{ \mathfrak
q }(t_\mathfrak q^{-r}\mathrm{d} \,
\mathrm{log}\,H(t_\mathfrak q,\bar b))+\alpha,$  for  some $\alpha \in
k$. Here, $aH(t_\mathfrak q,\bar b)\in B((t_\mathfrak q))$.

From   Lemma \ref{re}   and the above Remark, to  calculate
 $C(D^{ -r\cdot \mathfrak q})$  as a Cartier divisor on  $Y_R$,
 one can forget the vertical and horizontal components of $C$.
Thus,  this Cartier divisor is given,  up
to  units of $k[\epsilon] $, by
 $1+\epsilon \cdot  \mathrm{Res}_{ \mathfrak q }(t_\mathfrak q^{-r} \mathrm{d}\,
\mathrm{log}\,H(t_\mathfrak q,\bar b)). $
\end{rem}

\subsection{The principal part  ($\mathrm{mod} \,p$) for   certain meromorphic modular functions of level $4p^m$ with poles over the cusps }

By \cite{DR} VI,  2.3.1,  we have a section
 $s_n:\mathrm{Spec}({\Bbb Z}[\zeta_n])\to {  M}(n) $. The completion of
 ${ M}(n) $ along $s_n$  is identified with
 $\mathrm{Spec}({\Bbb
Z}[\zeta_n][[q^{1/n}]])\to { M}(n).  $ This morphism can be
obtained from the Tate elliptic curve with an $n$-level structure.
The completion along the cusps of  ${  M}(n)$  is a finite
number   of copies of $\mathrm{Spec}({\Bbb Z}[\zeta_n][[q^{1/n} ]])$. From
this, one deduces that the abelian   ramified covering
 $( {\mathcal C}_{A^0_m})_k\to
M(4)_k$  splits completely over the cusps
$M(4)_k\setminus M^0(4)_k$.

By tensoring $s_{4p^m}$ by  $\otimes_{R_{\mathfrak p}} k$,
one
obtains a section
 $  (s_{4p^m})_k:\mathrm{Spec}(k)\to   M (4p^m)_k   $.
We denote by  $\infty$   the geometric point of $
{  M}(4p^m)_k$ given by $( s_{4p^m})_k$ and we denote by $
\overline{q}$ the local parameter $q^{1/4p^m}$ ($\mathrm{mod} \, \mathfrak
p$).

 Recall that    $
M(4p^m)_k=\cup^{l-1}_{i=0} ({\mathcal C}_{A^i_m})_k.$ We denote by  $B$
the ring such that $\mathrm{Spec}(B)={  M}^0(4p^m)$.

Bearing in mind \cite{DR}  V, 4.19 and 4.20, we have an
isomorphism
 $$   \phi:  B_k [p(\sigma)^{-1}] \simeq (H_{A^0_m})_{S_0}\times \cdots \times (H_{A^{l-1}_m})_{S_{l-1}},$$
where  we denote    $\mathrm{Spec}(H_{A^i_m})=({\mathcal C}^0_{A^i_m})_k$ and $  S_0, \cdots, S_{l-1}$  are multiplicative systems  given
by the powers of the image of the polynomial $p(\sigma) $ in
  $H_{A^0_m}, \cdots , H_{A^{l-1}_m}$, respectively.
\begin{thm}\label{prin} For each $m\in {\Bbb N}$, there exists an element of $ B_k [p(\sigma)^{-1}]$
  such that its principal part  is
$$\sum_{i=0}^{d-1}
(\underset { \mathrm{deg}(r(\sigma))=d-1- i}  {\sum_{r(\sigma)\text{(monic)} }}
  \frac{1}{ ({\overline
q^{m p^{i\delta}})^{R^{-1}_{r(\sigma)}}} }) .$$
\end{thm}
\begin{proof} We assume that $\infty \in ({\mathcal C}_{A^j_m})_k$ and we denote by $\Sigma_{A^j_m}$ its fraction field and
by $\mu_m $ the idele whose entries are $1+\epsilon \overline
q^{-m} $ at $ \infty$ and $1$ elsewhere.
 Following the notation of section 5.2,  let $D^{-m\infty}$ be the Cartier divisor associated with  $\mu_m $.
 Now, if we consider the transpose of the trivial correspondence $\mathrm{deg}(\pi)^{-1}(D_k)$ of the Remark \ref{tri} then
  $\mathrm{deg}(\pi)^{-1}(^tD_k)(D^{-m\infty}) $ is the idele on $k[\epsilon]$
  $$  \prod_{i=0}^{d-1}
(\underset { \mathrm{deg}(r(\sigma))=d-1- i}  {\prod_{r(\sigma)\text{(monic)}} }
  \frac{1}{ ({\mu^{m{p^{i\delta}}} )^{R^{-1}_{r(\sigma)}}} })\in 1+\epsilon \frac{\Bbb A}{ O_{\Sigma_{A^j_m}}}.$$

Bearing in mind   Lemmas \ref{ep} and \ref{re},  we obtain an element $b_m\in (H_{A^j_m})_{S_{j}}$ such that
$1+\epsilon b_m\in 1+\epsilon \frac{\Bbb A}{ O_{\Sigma_{A^j_m}}} $ is the above idele. We finish by considering the element  $\phi^{-1}(0,\cdots,b_m,\cdots,0)\in B_k [p(\sigma)^{-1}] $.
\end{proof}

{\bf{Acknowlegments}} The author wishes to thank a referee by the hard work offered in improving the original article.
I also would like to express my gratitude to  Ricardo Perez Marco. I  am very grateful to Ricardo Alonso and  Jesus Mu\~noz-Diaz for their encouragement and help.

\vskip1.5truecm { \'Alvarez V\'azquez, Arturo}\newline {\it
e-mail: } aalvarez@usal.es


\vskip2truecm
\begin{thebibliography}{99}

\bibitem[An]{An} Anderson, G.  ''A two dimensional analogue of Stickelberg's theorem '' in: The Arithmetic of Function
Fields, ed.
 D.Goss, D.R. Hayes, W. de Gruyter, Berlin, 1992, pp.51-77.

\bibitem[BLR]{BLR} Bosch, S. Lutkebohmert, W. Raynaud, M.  ''N\'eron Models'', Springer-Verlag,    (1990).


\bibitem[BL]{BL} Beauville, A.  Laszlo, Y. ''Conformal blocks and
generalized theta functions'', Commun.Math.Phys. {\bf 164} (1994),
pp.385-489.


\bibitem[C]{C} Coleman, R.  ''On the Frobenius endomorphisms of the Fermat and Artin-Schreier curves'' in: The
Arithmetic of Function Fields, Proc. Amer. Math Soc, {\bf 102}
(1988), pp.463-466.



\bibitem[DR]{DR}Deligne, P., Rapoport, M. ''Les sch\'{e}mas de modules des courbes
elliptiques. In Modular functions of one variable''in: II (Proc.
Internat. Summer School, Univ. Antwerp,  1972) (Berlin,
1973), vol. 243, Lecture Notes in Math., Springer Verlag, pp.
143-316.

\bibitem[Ha]{Ha} Hayes, D. R.  ''Explicit class field theory in global function fields''  Estudies in algebra and number theory,   Adv. Math. Suplementary studies,    {\bf 6}
(1979), pp.173-217.

\bibitem[H]{H}  Husemoller, D. ''Elliptic Curves'',
Graduate Text in Mathematics, Springer-Verlag, New York, (1987).

\bibitem[I]{I}  Igusa, J. ''On the algebraic theory of elliptic modular functions'',
J. Math. Soc. Japan, 20, (1968) pp.96.

\bibitem[KM]{KM} Katz N.M., Mazur B. ''Arithmetic moduli of elliptic curves'', Ann. Math. Studies 108, Princeton University Press, Princeton, 1985.
350 (1972), pp.69-190.




\bibitem[K]{K}  Katz, N.M.  ''Higher congruences between modular forms'', Ann.
Math. 101 (1975), pp. 332--367.

\bibitem[L]{L} Lang, S. ''Introduction to modular forms''I, Springer-Verlag. Berlin. VII,
(1976).

\bibitem[M]{M} Mu\~{n}oz Diaz, J, S. ''Curso de teor{\'\i}a de funciones ''I,
Editorial Tecnos (1978).

\bibitem[R]{R} Raynaud, M.  ''Jacobienne des courbes modulaires  et op\'erateurs de Hecke''
Astérisque 196-197, (1991),
pp.9--25.

\bibitem[Ro]{Ro} Rosenlicht, M.  ''Generalized Jacobian varieties''
Ann. Math.      no. 59, (1954),
pp.505--530.

\bibitem[Se]{Se} Serre, J.P.  ''Congruences et formes modulaires''
Expos\'{e} 416, Seminaire N.Bourbaki 1971/72. LNM    no. 317,
pp.319--338.









\end{thebibliography}
\end{document}